\documentclass[a4paper]{amsart}
\usepackage{microtype}
\usepackage[utf8]{inputenc}

\usepackage{amsmath,amstext,amssymb,mathrsfs,amscd,amsthm,indentfirst, bbm}
\usepackage{amsfonts}
\usepackage{stmaryrd}

\usepackage[style=ieee-alphabetic, url=false, backref=true, doi=false, isbn=false, eprint=false]{biblatex}
\DeclareFieldFormat{url}{\textsc{url}\addcolon\space\url{#1}}

\usepackage{float}
\usepackage{adjustbox}
\usepackage{booktabs}
\usepackage{verbatim}
\usepackage{enumerate}
\usepackage{graphicx}
\usepackage{textcomp}
\usepackage{tikz,tikz-cd}
\usetikzlibrary{patterns}
\usepackage{ifthen}
\usepackage{url}

\usepackage{hyperref}
\PassOptionsToPackage{hyphens}{url}\usepackage{hyperref}
\hypersetup{
    colorlinks  = true,
    linkcolor   = [rgb]{.20,.60,.22},
    citecolor   = [rgb]{0,.40,.80},
    urlcolor    = [rgb]{0,.40,.80}
}
\usepackage[nameinlink,capitalize]{cleveref}
\usepackage{mathtools}
\usepackage{mleftright}
\usepackage{enumitem}

\usepackage{libertine}
\makeatletter
\renewcommand\th@plain{\slshape}
\makeatother

\usepackage[indentafter]{titlesec}
\titleformat{name=\section}{}{\thetitle.}{0.5em}{\centering\scshape}
\titleformat{name=\subsection}[runin]{}{\thetitle.}{0.5em}{\bfseries}[.]
\titleformat{name=\subsubsection}[runin]{\bfseries}{\thetitle.}{0.5em}{\itshape}[.]
\titleformat{name=\paragraph,numberless}[runin]{}{}{0em}{}[.]
\titlespacing{\paragraph}{0em}{0em}{0.5em}
\titleformat{name=\subparagraph,numberless}[runin]{}{}{0em}{}[.]
\titlespacing{\subparagraph}{0em}{0em}{0.5em}

\newcommand{\bC}{\mathbb C}

\newcommand{\bP}{\mathbb P}
\newcommand{\bQ}{\mathbb Q}

\newcommand{\bZ}{\mathbb Z}

\newcommand{\cC}{\mathcal C}

\newcommand{\cJ}{\mathcal J}

\newcommand{\cL}{\mathcal L}

\newcommand{\cP}{\mathcal P}

\newcommand{\cU}{\mathcal U}
\newcommand{\cX}{\mathcal X}

\newcommand{\fB}{\mathfrak B}
\newcommand{\fC}{\mathfrak C}
\newcommand{\fM}{\mathfrak M}

\newcommand{\fT}{\mathfrak T}
\newcommand{\fX}{\mathfrak X}

\newcommand{\lra}{\longrightarrow}

\newcommand{\Pic}{\mathrm{Pic}}

\newcommand{\Hol}{\mathrm{Hol}}
\newcommand{\Map}{\mathrm{Map}}

\newcommand{\Alg}{\mathrm{Alg}}

\newcommand{\an}{\mathrm{an}}

\newcommand{\Diff}{\mathrm{Diff}}

\DeclareMathOperator{\Spec}{Spec}
\DeclareMathOperator{\Proj}{Proj}
\DeclareMathOperator{\Sym}{Sym}

\newcommand{\catSch}{\mathsf{Sch}}
\newcommand{\catSchlft}{\mathsf{Sch}^\mathrm{lft}}

\newcommand{\catLCH}{\mathsf{LCH}}
\newcommand{\catGpd}{\mathsf{Groupoids}}
\newcommand{\opp}{\mathrm{op}}
\newcommand{\topp}{\mathrm{top}}

\newcommand{\stacks}[1]{\cite[\href{https://stacks.math.columbia.edu/tag/#1}{Tag #1}]{stacks-project}}

\NewDocumentCommand{\set}{somm}{%
   \IfNoValueTF{#2}
    {\IfBooleanTF{#1}{\{#3 \mid #4\}}{\mleft\{ #3 \mathrel{}\middle\vert\mathrel{} #4 \mright\}}}
    {\mathopen{#2\{}#3 \mathrel{}#2\vert\mathrel{} #4\mathclose{#2\}}}%
  }

\newtheorem{theorem}{Theorem}[section]
\newtheorem*{theorem*}{Theorem}
\newtheorem{theorem-in-progress}[theorem]{Theorem-in-progress}

\newtheorem*{conjecture*}{Conjecture}

\newtheorem{definition}[theorem]{Definition}
\newtheorem{lemma}[theorem]{Lemma}
\newtheorem{remark}[theorem]{Remark}

\newtheorem{example}[theorem]{Example}

\newtheorem{corollary}[theorem]{Corollary}

\newtheorem*{maintheorem*}{Main Theorem}

\title{Families of algebraic and continuous maps to $\bP^m$}
\author{Alexis Aumonier}
\address{Department of Mathematics, Stockholms universitet, 106 91 Stockholm, Sweden}
\email{alexis.aumonier@math.su.se}

\begin{document}

\begin{abstract}
We explain how results comparing the homology of spaces of algebraic and continuous maps to projective spaces can be leveraged to compare moduli stacks of families of algebraic and continuous maps.
\end{abstract}

\maketitle

\section{Introduction}

Let $\pi \colon \fX \to \fB$ be a smooth projective morphism between complex Deligne--Mumford stacks, which we think of as the universal family of a moduli of complex varieties. Our motivating examples are $\fC_g \to \fM_g$, the universal family of smooth genus $g \geq 2$ curves, and $\cX_d \to \cU_d$ the universal family of smooth projective degree $d$ hypersurfaces. In this note, we are interested in an associated moduli stack of families of maps to a fixed projective space whose functor of points is informally given by
\[
    T \longmapsto \left\{(p \colon T \to \fB, \ f \colon p^*\fX \to \bP^m)\right\}.
\]
Our motivation stems from extending the results of the case of a single variety \cite{Segal1979,Mostovy2010,Aumonier2024} to a parameterised setting. The pioneering work Segal \cite{Segal1979} showed that, for a smooth curve $C$, the inclusion
\[
    \Alg^d(C, \bP^m) \hookrightarrow \Map^d(C, \bP^m)
\]
of the space of algebraic maps of degree $d$ inside the space of continuous maps, induces an isomorphism in homology in a range of degree growing with $d$. An analogous result holds when $C$ is replaced by a projective space $\bP^n$ by work of Mostovoy \cite{Mostovy2010}, and the general case for any smooth projective variety $X$ is treated in \cite{Aumonier2024}. One may ask if similar results hold when $C$, or more generally when $X$, varies in a moduli. As we understand it, this question was considered for curves varying in $\fM_g$ by Ayala in \cite{Ayala2008} but unfortunately that preprint contained a mistake, as explained in \cite[Remark~2.0.1]{Ayala2008}. We use our more recent understanding of spaces of algebraic maps to give a proof of the main theorem of \cite{Ayala2008} (whose statement remains valid), as a special case of a general theorem for families of maps from smooth projective varieties.

For higher dimensional varieties (not curves), the degree of a map to projective space is a class in the second cohomology group of the source. For families of maps, the natural notion is a continuous choice of such classes living in the sheaf cohomology group $H^0(\fB^\topp; R^2\pi_*\underline{\bZ})$ (we have written $(-)^\topp$ for the underlying topological stack, see \cref{def:topologification}). Given a class $\alpha$ in this cohomology group, we construct an algebraic (resp. topological) stack $\Alg_\fB^\alpha(\fX, \bP^m)$ (resp. $\Map_\fB^\alpha(\fX, \bP^m)$) parameterising families of algebraic (resp. continuous) maps to $\bP^m$ of fibrewise degree $\alpha$. We defer the precise definitions to \cref{sec:definitions}. As any algebraic map is in particular continuous, there is a natural morphism from (the underlying topological stack of) the first stack to the second. Our theorem concerns its effect on singular integral homology:
\begin{maintheorem*}[See \cref{theorem:precise-main} for a precise version]
Let $n$ be the dimension of the fibres of $\fX \to \fB$, and $\alpha \in H^0(\fB^\topp; R^2\pi_*\underline{\bZ})$. Then the natural map
\[
    \Alg_\fB^\alpha(\fX, \bP^m)^\topp \lra \Map_\fB^\alpha(\fX, \bP^m)
\]
induces an isomorphism in homology in the range of degrees $* < (2m - 2n +1) \cdot d(\fX,\alpha) + 2m - 2n - 3$, where $d(\fX, \alpha)$ is the integer defined in \cref{def:uniform-d-alpha}.
\end{maintheorem*}

The number $d(\fX, \alpha)$ is hard to compute in general, though it can be made arbitrarily big up to changing $\alpha$, see e.g. \cite[Section~3]{Aumonier2024} for some properties. In good cases however it is easy to understand: for the moduli of maps of degree $d$ from families of curves of genus $g \geq 2$, it is $2g-d$ (see \cite[Lemma~3.4]{Aumonier2024}). We have unravelled the consequences of our main theorem in this case in \cref{sec:teichmuller}, using Teichmüller theory in a spirit closer to the preprint \cite{Ayala2008}.

\subsection{Consequences in rational homology}

The main benefit of our theorem (or any theorem of the flavour of Segal's theorem \cite{Segal1979}) is the possibility of computing the homology (though only in a range) of a very algebraic object by means of homotopical methods. For families of maps our theorem is most advantageously combined with results on the cohomology of $\fB$, often themselves obtained by homotopical methods. Let us recall the following example from \cite{Ayala2008} to illustrate the kind of results one can obtain.
\begin{corollary}
Let $g \geq 2$, and let $\fC_g \to \fM_g$ be the universal family of smooth genus $g$ curves. Let $K$ be the graded vector space over $\bQ$ generated by the set $\{k_i\}$ in degrees $|k_i| = 2i$. Write $H^*(\bP^m;\bQ) = \bQ[x]/(x^{m+1})$ with $|x| = 2$. Then, for $d \geq 2g$ and $m \geq 1$, and trough degrees $* < \min((2m - 1)(d-2g) + 2m - 5, (g-5)/2)$ we have an isomorphism of graded rings
\[
    H_*(\Alg_{\fM_g}^d(\fC_g,\bP^m); \bQ) \cong \Sym\big( (K \otimes H^*(\bP^m;\bQ))_+ \big)
\]
where $\Sym$ is the free graded commutative algebra and $(-)_+$ is the positively graded summand.
\end{corollary}
\begin{proof}
This is a consequence of a result of Cohen and Madsen \cite{CohenMadsen2009} on the stable homology of the moduli of curves equipped with a map to a background space. See the introduction of \cite{Ayala2008} for more details.
\end{proof}

\subsection*{Acknowledgements} I would like to thank Eric Ahlqvist for help with stacks. I was supported by Dan Petersen's Wallenberg Scholar fellowship.

\section{Recollections on stacks}

In this note, we shall work with both algebraic and topological stacks. We denote by $\catSch$ the category of complex schemes, by $\catSchlft$ the subcategory of schemes locally of finite type, and by $\catLCH$ the category of locally compact Hausdorff spaces. As usual, an algebraic stack is a stack on the site $\catSch$ with the étale topology. A topological stack is a stack on the site $\catLCH$ with the open cover topology. Let us recall from \cite[Section~4]{Jansen2024} how to obtain a topological stack from a Deligne--Mumford stack:
\begin{definition}\label{def:topologification}
Let $\fM$ be a Deligne--Mumford stack. Its \emph{underlying topological stack} $\fM^\topp$ is the topological stack obtained by sheafifying the presheaf
\begin{align*}
    \catLCH^\opp &\lra \catGpd \\
    T &\longmapsto \overline{\fM}\big(\Spec(\Map(T, \bC))\big)
\end{align*}
where $\overline{\fM}$ is the left Kan extension of $\fM$ along the inclusion $\catSchlft \hookrightarrow \catSch$.
\end{definition}
\begin{remark}
Contrary to \cite[Section~4]{Jansen2024}, we prefer to work with schemes \emph{locally} of finite type and drop the second countability assumption on our category of topological spaces. This is inconsequential as the analytification functor is defined on complex schemes locally of finite type (see \cite[Théorème et définition~1.1]{Raynaud1971}) and the category $\catLCH$ admits fibre products (see \stacks{09WZ} and \cite[Remark~2.4]{Jansen2024}). We also work with stacks instead of $\infty$-stacks, but sheafification preserves this truncatedness.
\end{remark}

\begin{remark}
If $X$ is a scheme locally of finite type, its underlying topological space $X^\topp$ constructed by \cref{def:topologification} agrees with its analytification $X^\an$, see e.g. \cite[Lemma~4.6]{Jansen2024}.
\end{remark}

We say that a morphism of topological stacks $T \to \fT$ with $T \in \catLCH$ is \emph{étale} if it is representable and a local homeomorphism. We can associate a homotopy type to a topological stack as follows:
\begin{definition}\label{def:homotopy-type}
Let $\fT$ be a topological stack with an étale presentation $T \to \fT$. Its \emph{homotopy type} is the homotopy type of the geometric realisation of the semi-simplicial space
\[
    [p] \longmapsto T \times_\fT \cdots \times_\fT T \quad \text{($p$ times).}
\]
The homotopy type of a Deligne--Mumford stack is defined to be that of its underlying topological stack.
\end{definition}

\begin{remark}
Given an étale presentation $U \to \fX$ of a Deligne-Mumford stack we obtain an étale presentation (i.e. a surjective representable local homeomorphism) of the underlying topological stack $U^\topp \to \fX^\topp$, see \cite[Proposition~4.10]{Jansen2024}. Hence \cref{def:homotopy-type} agrees with the usual definition in the literature, see e.g. \cite{Ebert2009}.
\end{remark}

Any topological invariant of a topological stack, e.g. its singular homology, is then defined via its underlying homotopy type. In particular, the following observation will be useful.
\begin{lemma}\label{lemma:semisimplicial-comparison}
Let $X_\bullet \to Y_\bullet$ be a morphism of semi-simplicial spaces. Suppose that there exists an integer $N \geq 0$ such that $X_p \to Y_p$ is a homology isomorphism in degrees $* < N$ for all $p \geq 0$. Then $\lVert X_\bullet \rVert \to \lVert Y_\bullet \rVert$ is a homology isomorphism in degrees $* < N$.
\end{lemma}
\begin{proof}
There is a skeletal first quadrant spectral sequence of signature
\[
    E^1_{p,q} = H_q(X_p; \bZ) \implies H_{p+q}(\lVert X_\bullet \rVert; \bZ)
\]
and likewise for $Y_\bullet$ (see e.g. \cite[Section~1.4]{EbertRandalWilliams2019}). The lemma follows by comparing the spectral sequences.
\end{proof}

\section{Families of maps}\label{sec:definitions}

Let $\fB$ be a Deligne--Mumford stack, and $\fX \to \fB$ be a smooth projective morphism representable by schemes. We shall say that a map $X \to B$ between schemes, or more generally algebraic spaces or stacks, (resp. topological spaces) is an algebraic (resp. topological) $\fX$-family over $B$ if it is a pulled back family $X = f^*\fX$ for some $f \colon B \to \fB$ (resp. $X = f^*\fX^\topp$ and $f \colon B \to \fB^\topp$). In this section, we shall first explain how to define the degree of a map from an $\fX$-family, and then properly define the stacks appearing in our main theorem.

\subsection{Degree of a family of maps}
We now explain what we mean by the degree of a map $\fX \to \bP_\fB^m$ over $\fB$, which we think of as a family of maps parametrised by the base $\fB$. First recall that for a smooth projective complex variety $X$, the degree of a map $f \colon X \to \bP^m$ is defined to be the cohomology class $f^*(h) \in H^2(X; \bZ)$ where $h \in H^2(\bP^m; \bZ)$ is the positive generator. When thinking of $\fX \to \fB$ as a family of varieties $X_b$ parameterised by $b \in \fB$, we only have access to the groups $H^2(X_b;\bZ)$ up to monodromy automorphisms. Hence only the classes that are invariant under monodromy give well-defined fibrewise degrees. This motivates the following (we write $\underline{A}$ for the constant sheaf associated to an abelian group $A$):
\begin{definition}\label{def:fibrewise-degree}
Let $p \colon X \to B$ be a topological $\fX$-family and let $g \colon X \to \bP^m$ be a continuous map. There is an induced map of sheaves $\underline{H^2(\bP^m;\bZ)} \to R^2p_*\underline{\bZ}$ on $B$, and we define the fibrewise degree of $g$ to be the image $\deg(g) \in H^0(B; R^2p_*\underline{\bZ})$ of the positive generator $h \in H^0(B, \underline{H^2(\bP^m;\bZ)})$ in sheaf cohomology.
\end{definition}

\begin{remark}\label{remark:fibre-degree}
This definition recovers the usual degree when $B$ is a point. More generally, for a map $g \colon X \to \bP^m$ and a point $b \in B$, we obtain a map on the fibre $g_b \colon X_b \to \bP^m$ of degree $b^*(\deg(g)) \in H^0(\{b\}, H^2(X_b;\bZ)) \cong H^2(X_b;\bZ)$ by base change along the inclusion $\{b\} \subset B$. This justifies our choice of definition.
\end{remark}

More generally, our \cref{def:fibrewise-degree} makes sense for the universal family $\pi \colon \fX^\topp \to \fB^\topp$ provided one uses the theory of sheaves on topological stacks and their cohomology. In particular, a class in $H^0(\fB^\topp; R^2\pi_*\underline{\bZ})$ gives rise, by base change, to a pulled back class in $H^0(B; R^2p_*\underline{\bZ})$ for any $\fX$-family $p \colon X \to B$.

\subsection{Ampleness}
We can only expect to fruitfully compare algebraic and continuous maps when their degree is ``big enough". This can be quantified using ampleness qualities of line bundles, or intuitively: how much sections of a given line can interpolate between points.
\begin{definition}
Let $\cL$ be a line bundle on a smooth projective variety $X$ and $d \geq 0$ be an integer. We say that $\cL$ \emph{separates $d$ points} if for any choice of $d$ distinct points $x_1,\dotsc,x_d \in X$ the evaluation map
\[
    H^0(X, \cL) \lra \bigoplus_{i=1}^d \cL|_{x_i}, \quad s \longmapsto (s(x_1),\dotsc,s(x_d))
\]
is surjective. We write $d(X, \cL) \geq 0$ for the biggest integer $d$ such that $\cL$ separates $(d+1)$ points.
\end{definition}

As in \cite{Aumonier2024}, we will need to consider line bundles with the same first Chern class but different holomorphic structures, as well as require those line bundles to be ample enough (see \cite[Section~2]{Aumonier2024}). We thus adopt the following notation:
\begin{definition}
Let $X$ be a smooth projective variety and $\alpha \in H^2(X;\bZ)$. We define
\[
    d(X, \alpha) = \begin{cases}
        \inf\limits_{[\cL] \in \Pic^\alpha(X)} \ d(X,\cL) & \text{if } \Pic^\alpha(X) \neq \emptyset \text{ and } \alpha - c_1(K_X) \text{ is ample}, \\
        -\infty & \text{else}.
        \end{cases}
\]
\end{definition}
The ampleness condition is chosen to ignore special algebraic phenomena which appears when considering algebraic maps of low degree. Now, for a family of varieties and a given degree, our main theorem relies on a bound uniform on all possible varieties appearing in the family. We package it into the following:
\begin{definition}\label{def:uniform-d-alpha}
Let $\pi \colon \fX \to \fB$ be a smooth projective morphism representable by schemes, and $\alpha \in H^0(\fB^\topp; R^2\pi_*\underline{\bZ})$. We define
\[
    d(\fX, \alpha) = \inf\limits_{b \in \fB(\bC)} d(X_b,\alpha_b)
\]
where $X_b = \Spec(\bC) \times_\fB \fX$ is the variety corresponding to $b \in \fB(\bC)$ and $\alpha_b$ is the pullback of $\alpha$ along $b$ (see \cref{remark:fibre-degree}).
\end{definition}

\subsection{Stacks of families of maps}

We define the two stacks of interest in this note. Let $\alpha \in H^0(\fB^\topp, R^2\pi_*\underline{\bZ})$ be a degree. We begin with the algebraic stack of families of algebraic maps.

\begin{definition}
Let $\Alg_\fB^\alpha(\fX, \bP^m)$ be the following stack on the site $\catSch$. An object of $\Alg_\fB^\alpha(\fX, \bP^m)(T)$, $T \in \catSch$, is a pair $(f,g)$ of an $f \in \fB(T)$ and a map $g \colon f^*\fX \to \bP^m$ of fibrewise degree $\alpha$. An isomorphism $(f,g) \to (f',g')$ is an isomorphism $f \to f'$ in $\fB$ inducing an isomorphism $\varphi$ on the pulled back families such that $g' \circ \varphi = g$.
\end{definition}

We offer two motivating examples.
\begin{example}
Let $g \geq 2$ and let $\pi \colon \fC_g \to \fM_g$ be the universal family of smooth genus $g$ curves. As any such curve is canonically oriented the sheaf $R^2\pi_*\underline{\bZ}$ is isomorphic to the constant sheaf $\underline{\bZ}$. Hence a degree is simply an integer $d \in \bZ$. In fact, by Poincaré duality, a map $\fC_g \to \bP^m$ of degree $d$ is a morphism such that its restriction to each fibre $f \colon C \to \bP^m$ satisfies $f_*[C] = d \cdot [\bP^1] \in H_2(\bP^m;\bZ)$.
\end{example}

\begin{example}
Let $X$ be a smooth projective variety of dimension at least $3$ and $\cL$ be a very ample line bundle on it. Let $|\cL|^\mathrm{ns} \subset |\cL|$ be the open subset of the linear system which parameterises non-singular hypersurfaces. Write $\pi \colon \cU \to |\cL|^\mathrm{ns}$ for the universal family. By the Lefschetz hyperplane theorem the sheaf $R^2\pi_*\underline{\bZ} \cong \underline{H^2(X;\bZ)}$ is constant. Given $\alpha \in H^2(X;\bZ)$, a map $\cU \to \bP^m$ of degree $\alpha$ restricts to maps $Z \to \bP^m$ of degree $\iota^*\alpha$ where $\iota \colon Z \hookrightarrow X$ are elements of $|\cL|^\mathrm{ns}$.
\end{example}

The definition of the topological stack of families of continuous maps is similar:
\begin{definition}
Let $\Map_\fB^\alpha(\fX,\bP^m)$ be the following stack on the site $\catLCH$. An object of $\Map_\fB^\alpha(\fX,\bP^m)(T)$, $T \in \catLCH$, is a pair $(f,g)$ of an $f \in \fB^\topp(T)$ and a continuous map $g \colon f^*\fX^\topp \to \bP^m$ of fibrewise degree $\alpha$. An isomorphism $(f,g) \to (f',g')$ is an isomorphism $f \to f'$ in $\fB^\topp$ inducing an isomorphism $\varphi$ on the pulled back families such that $g' \circ \varphi = g$.
\end{definition}

In the rare cases when $\Alg_\fB^\alpha(\fX, \bP^m)$ is a scheme, there is certainly a map from its analytification to $\Map_\fB^\alpha(\fX,\bP^m)$ simply given by noticing that an algebraic map is continuous. In general, we need to recall the construction of \cref{def:topologification}: let $T \in \catLCH$ and consider the morphism of prestacks
\[
    \overline{\Alg_\fB^\alpha(\fX, \bP^m)}(\Spec(\Map(T,\bC))) \to \Alg_\fB^\alpha(\fX, \bP^m)(\Spec(\Map(T,\bC))) \to \Map_\fB^\alpha(\fX,\bP^m)(T)
\]
where the first arrow is the canonical morphism from the left Kan extension, and the second arrow sends a family of maps $(f,g)$ to the family $(f \circ \phi, g \circ \varphi)$ where $\varphi \colon T \to \Spec(\Map(T, \bC))$ is the canonical morphism and $\phi \colon \varphi^*f^*\fX^\topp \to f^*\fX^\topp$ is the induced morphism from the pulled back family:
\[
\begin{tikzcd}
\varphi^*f^*\fX^\topp \arrow[r, "\phi"] \arrow[d] & f^*\fX^\topp \arrow[r, "g"] \arrow[d] & \bP^m \\
T \arrow[r, "\varphi"]  & {\Spec(\Map(T,\bC))}  &
\end{tikzcd}
\]
We have used the same letter $g$ for the algebraic map $\fX \to \bP^m$ and its induced continuous map after analytification. By universal property of sheafification we obtain a morphism of topological stacks
\[
    \Alg_\fB^\alpha(\fX, \bP^m)^\topp \lra \Map_\fB^\alpha(\fX,\bP^m)
\]
which we shall refer as the natural morphism (and we will do the same for its restriction to $\fX$-families). Let us finally properly state our main theorem:
\begin{theorem}\label{theorem:precise-main}
Let $\fB$ be a Deligne--Mumford stack locally of finite type. Let $\pi \colon \fX \to \fB$ be a smooth projective morphism representable by schemes. Suppose that all the fibres of $\pi$ have the same complex dimension $n$. Let $\alpha \in H^0(\fB^\topp; R^2\pi_*\underline{\bZ})$. Then the natural map
\[
    \Alg_\fB^\alpha(\fX, \bP^m)^\topp \lra \Map_\fB^\alpha(\fX,\bP^m)
\]
induces an isomorphism in singular integral homology in the range of degrees $* < (2m - 2n +1) \cdot d(\fX,\alpha) + 2m - 2n - 3$.
\end{theorem}

\section{Proof of the main theorem}

In this section we fix a Deligne--Mumford stack $\fB$ locally of finite type, and a smooth projective morphism $\pi \colon \fX \to \fB$ representable by schemes. We assume that all the fibres of $\pi$ have the same complex dimension $n$. We begin with a few lemmas and shall end this section by proving \cref{theorem:precise-main}.

As usual with stacks, the proof starts with a key lemma about schemes and then uses various reductions to deduce the general statement on stacks. The main difference between the case of a single variety treated in \cite{Aumonier2024} and the family case is the existence of a Poincaré line bundle. It allows to construct the stack of algebraic maps as a scheme, as done in the following lemma:
\begin{lemma}\label{lemma:comparison-with-section}
Let $X \to B$ be an algebraic $\fX$-family between schemes locally of finite type. Suppose that it admits a section. Then the natural morphism
\[
    \Alg_B^\alpha(X, \bP^m)^\topp \to \Map_B^\alpha(X, \bP^m)
\]
induces an isomorphism on singular homology in the range of degrees $* < (2m - 2n +1) \cdot d(\fX,\alpha) + 2m - 2n - 3$.
\end{lemma}
\begin{proof}
We proceed as in \cite[Section~2 and Section~4.3]{Aumonier2024}. As the family $X \to B$ has a section, the relative Picard variety $\Pic^\alpha(X/B)$ exists and represents the relative Picard functor (see e.g. \cite[Exercise~4.3]{Kleiman2005}). In particular, we may choose a Poincaré line bundle $\cP$ on $\Pic^\alpha(X/B) \times_B X$. The argument of \cite[Section~2]{Aumonier2024} then shows that $\Alg_B^\alpha(X, \bP^m)$ is a scheme which may be constructed as an open subscheme
\[
    \Alg_B^\alpha(X, \bP^m) \subset \underline{\Proj}_{\Pic^\alpha(X/B)}\big(\Sym \ (q_*\cP^{\oplus m+1})^\vee\big)
\]
where $q \colon \Pic^\alpha(X/B) \times_B X \to \Pic^\alpha(X/B)$ is the first projection. Thus the projection $\Alg_B^\alpha(X, \bP^m) \to B$ is the restriction of a fibre bundle to an open subset of the total space. On the other hand, it is straightforward to check that $\Map_B^\alpha(X, \bP^m) \to B$ is a fibre bundle. The lemma finally follows from \cite[Lemma~4.7]{Aumonier2024}, whose assumptions are verified above and using \cite[Theorem~A]{Aumonier2024}.
\end{proof}

The following result is the main reduction lemma.
\begin{lemma}\label{lemma:comparison-with-schemes}
Let $X \to B$ be an algebraic $\fX$-family between schemes locally of finite type. Then the natural morphism
\[
    \Alg_B^\alpha(X, \bP^m)^\topp \to \Map_B^\alpha(X, \bP^m)
\]
induces an isomorphism on singular homology in the range of degrees $* < (2m - 2n +1) \cdot d(\fX,\alpha) + 2m - 2n - 3$.
\end{lemma}
\begin{proof}
The morphism $X \to B$ is smooth, so there exists an étale surjective morphism $U \to B$ such that $X_U = U \times_B X \to U$ admits a section \cite[17.16.3.(ii)]{EGAIV}. More generally, for any integer $p \geq 1$ we write $U_p = U \times_B \cdots \times_B U$ for the $p$-fold product, and observe that $X_{U_p} \to U_p$ admits a section as well. As observed in the proof of \cref{lemma:comparison-with-section} above, the stack $\Alg_U^\alpha(X_U, \bP^m)$ is in fact a scheme and the pullback square
\[
\begin{tikzcd}
{\Alg_U^\alpha(X_U, \bP^m)} \arrow[r] \arrow[d] & {\Alg_B^\alpha(X, \bP^m)} \arrow[d] \\
U \arrow[r]  & B
\end{tikzcd}
\]
shows that it is an étale presentation of $\Alg_B^\alpha(X, \bP^m)$. Hence the homotopy type of this latter stack can be computed as the geometric realisation of the semi-simplicial space
\[
    [p] \longmapsto \Alg_U^\alpha(X_U, \bP^m) \times_{\Alg_B^\alpha(X, \bP^m)} \cdots \times_{\Alg_B^\alpha(X, \bP^m)} \Alg_U^\alpha(X_U, \bP^m) \qquad \text{($p$ times)}.
\]
Now, simple manipulations of pullbacks identify this fibre product as
\[
    \Alg_U^\alpha(X_U, \bP^m)^{\times_{\Alg_B^\alpha(X, \bP^m)} p} \cong \Alg_{U_p}^\alpha(X_{U_p}, \bP^m).
\]
On the other hand, the homotopy type of $\Map_B^\alpha(X, \bP^m)$ can be computed by a similar semi-simplicial space where algebraic maps are replaced by continuous ones. We thus obtain a morphism of semi-simplicial spaces
\[
    \Alg_{U_\bullet}^\alpha(X_{U_\bullet}, \bP^m) \lra \Map_{U_\bullet}^\alpha(X_{U_\bullet}, \bP^m)
\]
which is a levelwise homology isomorphism in a range by \cref{lemma:comparison-with-section}. The proof finishes by invoking \cref{lemma:semisimplicial-comparison}.
\end{proof}

As usual when working with stacks, we might run into situations involving algebraic spaces. We thus state the following further reduction for convenience.
\begin{lemma}\label{lemma:comparison-with-spaces}
Let $X \to B$ be an algebraic $\fX$-family between algebraic spaces locally of finite type. Then the natural morphism
\[
    \Alg_B^\alpha(X, \bP^m)^\topp \to \Map_B^\alpha(X, \bP^m)
\]
induces an isomorphism on singular homology in the range of degrees $* < (2m - 2n +1) \cdot d(\fX,\alpha) + 2m - 2n - 3$.
\end{lemma}
\begin{proof}
Pick an étale presentation $U \to B$ of the algebraic space $B$ and proceed as in the proof of \cref{lemma:comparison-with-schemes}.
\end{proof}

We are now ready to prove our main theorem.
\begin{proof}[\proofname\space of \cref{theorem:precise-main}]
Let $U \to \fB$ be an étale presentation and write $U_p$ for the $p$-fold product $U \times_\fB \cdots \times_\fB U$ (which is a priori only an algebraic space). The homotopy type of the stack $\Alg_\fB^\alpha(\fX, \bP^m)$ can be computed as the geometric realisation of the semi-simplicial space $\Alg_{U_\bullet}^\alpha(\fX_{U_\bullet}, \bP^m)$ and likewise for its continuous analogue $\Map_\fB^\alpha(\fX, \bP^m)$. By \cref{lemma:comparison-with-spaces}, the natural map
\[
    \Alg_{U_\bullet}^\alpha(\fX_{U_\bullet}, \bP^m) \lra \Map_{U_\bullet}^\alpha(\fX_{U_\bullet}, \bP^m)
\]
is a levelwise homology isomorphism in a range. The theorem then follows by applying \cref{lemma:semisimplicial-comparison}.
\end{proof}

\section{Consequences for families of maps from curves}\label{sec:teichmuller}

Let $g \geq 2$, and denote by $\pi \colon \fC_g \to \fM_g$ the universal family of smooth genus $g$ curves. Choose once and for all a smooth genus $g$ closed real surface $\Sigma_g$, and let $\Diff^+(\Sigma_g)$ be the topological group of orientation preserving diffeomorphism of $\Sigma_g$. Let $\cJ_g$ be the space of complex structures on $\Sigma_g$ (see \cite[Section~1.1]{Ayala2008} for a definition). The group $\Diff^+(\Sigma_g)$ acts on $\cJ_g$ and it is well-known that the homotopy type of $\fM_g$ is that of the homotopy quotient $\cJ_g \sslash \Diff^+(\Sigma_g)$. We will analogously describe the homotopy types of the stacks of algebraic and continuous maps. Let
\[
    \cJ_g^d(\bP^m) = \set{(J,f)}{J \in \cJ_g, \ f \in \Hol^d((\Sigma_g,J), \bP^m)} \subset \cJ_g \times \Map^d(\Sigma_g, \bP^m).
\]
It is a standard consequence of Chow's theorem that the spaces of algebraic and holomorphic maps agree in the present case. Accordingly, we can identify the homotopy type of the stack of algebraic maps
\[
    |\Alg_{\fM_g}^d(\fC_g, \bP^m)| \simeq \cJ_g^d(\bP^m) \sslash \Diff^+(\Sigma_g).
\]
On the other hand, using contractibility of $\cJ_g \simeq *$, we may also compute
\[
    |\Map_{\fM_g}^d(\fC_g,\bP^m)| \simeq (\cJ_g \times \Map^d(\Sigma_g, \bP^m)) \sslash \Diff^+(\Sigma_g) \simeq \Map^d(\Sigma_g, \bP^m) \sslash \Diff^+(\Sigma_g).
\]
We give a more direct proof of our main theorem specialised in this case:
\begin{theorem}\label{theorem:ayala}
Let $d \geq 2g+1$ and $m \geq 1$. Then the map
\[
    \cJ_g^d(\bP^m) \sslash \Diff^+(\Sigma_g) \lra (\cJ_g \times \Map^d(\Sigma_g, \bP^m)) \sslash \Diff^+(\Sigma_g)
\]
induces an isomorphism in homology in the range $* < (2m-1)(d-2g) + 2m - 5$.
\end{theorem}
\begin{proof}
By comparing the Serre spectral sequences of the homotopy orbits constructions, it suffices to show that the map
\[
    \cJ_g^d(\bP^m) \lra \cJ_g \times \Map^d(\Sigma_g, \bP^m)
\]
induces an isomorphism in homology in the claimed range. This is also the strategy employed in \cite{Ayala2008}. We claim that the projection map $\cJ_g^d(\bP^m) \to \cJ_g$ is locally on the base a microfibration obtained as the restriction of a fibre bundle to an open subset. Indeed, there is a universal family of curves $\cC_g \to \cJ_g$ which locally admits a section. Hence, the argument in the proof of \cref{lemma:comparison-with-section} explains how to construct, locally, the space of algebraic maps as an open subset of a projective bundle. We may then proceed as in \cref{lemma:comparison-with-schemes} to finish the proof.
\end{proof}

\begin{remark}
The observant reader might notice that our range of homology isomorphism differs from that stated in \cite{Ayala2008}. This is only an artifact of Segal's range from \cite{Segal1979} which is not always optimal, and we have chosen to use that of \cite{Aumonier2024} (though not believed to be optimal either).
\end{remark}

\begin{remark}
For consequences in homology of \cref{theorem:ayala}, see the introduction of \cite{Ayala2008}, which remains valid as it only relies on the final statement of the main theorem.
\end{remark}

\printbibliography

\end{document}